\documentclass[12pt]{amsart}
\usepackage[colorlinks=true,hyperindex,citecolor=Green,linkcolor=Blue]{hyperref}
\usepackage[top=1in, bottom=1in, left=1.1in, right=1.1in]{geometry}
\usepackage{amsmath}
\usepackage{amsfonts}
\usepackage{amssymb}
\usepackage{amsthm}
\usepackage[all]{xy}
\usepackage{mathrsfs}
\usepackage{enumitem} 
\usepackage[T1]{fontenc}
\usepackage{color}

\makeatletter
\def\namedlabel#1#2{\begingroup
#2%
\def\@currentlabel{#2}%
\phantomsection\label{#1}\endgroup
}
\makeatother

\theoremstyle{theorem} 
\newtheorem{theorem}{Theorem}[section]

\newtheorem{lemma}[theorem]{Lemma}
\newtheorem{proposition}[theorem]{Proposition}

\newtheorem{theoremx}{Theorem}

\theoremstyle{definition} 
\newtheorem{definition}[theorem]{Definition}

\newtheorem{remark}[theorem]{Remark}
\numberwithin{equation}{subsection}
\newtheorem{notation}[theorem]{Notation}

\definecolor{blue-violet}{rgb}{0.54, 0.17, 0.89}
\definecolor{Blue}{rgb}{0.01, 0.28, 1.0}
\definecolor{gGreen}{rgb}{0.2, 0.8, 0.2}
\definecolor{Green}{rgb}{0.04, 0.85, 0.32}

\makeatletter
\def\@tocline#1#2#3#4#5#6#7{\relax
  \ifnum #1>\c@tocdepth 
  \else
    \par \addpenalty\@secpenalty\addvspace{#2}%
    \begingroup \hyphenpenalty\@M
    \@ifempty{#4}{%
      \@tempdima\csname r@tocindent\number#1\endcsname\relax
    }{%
      \@tempdima#4\relax
    }%
    \parindent\z@ \leftskip#3\relax \advance\leftskip\@tempdima\relax
    \rightskip\@pnumwidth plus4em \parfillskip-\@pnumwidth
    #5\leavevmode\hskip-\@tempdima
      \ifcase #1
       \or\or \hskip 1.9em \or \hskip 2em \else \hskip 3em \fi%
      #6\nobreak\relax
    \dotfill\hbox to\@pnumwidth{\@tocpagenum{#7}}\par
    \nobreak
    \endgroup
  \fi}
\makeatother

\newcommand{\NN}{\mathbb{N}}
\newcommand{\RR}{\mathbb{R}}
\newcommand{\ZZ}{\mathbb{Z}}
\newcommand{\QQ}{\mathbb{Q}}
\newcommand{\KK}{\mathbb{K}}

\newcommand{\CC}{\mathbb{C}}
\newcommand{\m}{\mathfrak{m}}

\newcommand{\Spec}{\operatorname{Spec}}

\newcommand{\ch}{\operatorname{char}}

\newcommand{\Nash}{\operatorname{Nash}}

\newcommand{\Hilb}{\operatorname{Hilb}}

\newcommand{\Oxx}{{\mathcal O}_{X,x}}
\newcommand{\Su}{\KK[x^{a_1},\ldots,x^{a_s}]}
\newcommand{\Po}{\KK[x_1,\ldots,x_d]}
\newcommand{\Poy}{\KK[y_1,\ldots,y_s]}
\newcommand{\Poz}{\ZZ_p[y_1,\ldots,y_s]}

\newcommand{\Na}{\operatorname{Nash}_n(X)}
\newcommand{\Nn}{\overline{\Nash_n(X)}}

\newcommand{\V}{\mathbf{V}}

\newcommand{\GrFan}{\operatorname{GF}}

\newcommand{\GF}{\GrFan(J_n)}

\newcommand{\sd}{\check{\sigma}}
\newcommand{\Si}{\mbox{Sing}}

\newcommand{\sz}{\sigma_{\ZZ}}
\newcommand{\Suc}{\CC[u,u^3v^4,uv]}
\newcommand{\Suk}{\KK[u,u^3v^4,uv]}
\newcommand{\Suz}{\ZZ[u,u^3v^4,uv]}
\newcommand{\Jn}{\langle u-1,u^3v^4-1,uv-1 \rangle^{n+1}}
\newcommand{\Gn}{\mathbb{G}_n}

\newcommand{\Pnn}{\mathcal{P}_n}
\newcommand{\Pnu}{\mathcal{P}_{n-1}}
\newcommand{\Cc}{C_{\mathbb{G}_n^{(0)}}}
\newcommand{\Cp}{C_{\mathbb{G}_n^{(p)}}}

\newcommand{\lm}{\operatorname{lm}}
\newcommand{\lc}{\operatorname{lc}}
\newcommand{\lt}{\operatorname{lt}}
\newcommand{\init}{\operatorname{in}}
\newcommand{\supp}{\operatorname{supp}}

\newcommand{\pol}{\operatorname{pol}}
\newcommand{\tor}{\mathbb{T}}

\begin{document}

\title[Higher Nash blowups of normal toric varieties]{Higher Nash blowups of normal toric varieties in prime characteristic}

\author[D. Duarte]{Daniel Duarte$^1$}
\address{CONACyT-Universidad Aut\'onoma de Zacatecas, Zacatecas, Zac., M\'exico}
\email{aduarte@uaz.edu.mx}

\author[L. N\'u\~nez-Betancourt]{Luis N\'u\~nez-Betancourt${^2}$}
\address{Centro de Investigaci\'on en Matem\'aticas, Guanajuato, Gto., M\'exico}
\email{luisnub@cimat.mx}

\thanks{{$^1$}Partially supported by CONACyT grant 287622.}

\thanks{{$^2$}Partially supported by CONACyT grant 284598 and C\'atedras Marcos Moshinsky.}


\maketitle
\setcounter{tocdepth}{1}

\begin{abstract}
We prove that the higher Nash blowup of a normal toric variety defined over a field of positive characteristic is an isomorphism if and only if it is non-singular. We also extend a result by R. Toh-Yama which shows that higher Nash blowups do not give a one-step resolution of certain toric surface.
These results were previously  known only in characteristic zero. 
\end{abstract}


\section{Introduction}

The Nash blowup of order $n$ is a modification of an algebraic variety that replaces singular points by limits of infinitesimal neighborhoods of order $n$ of non-singular points. The main goal for this generalization of the usual Nash blowup was to investigate whether this modification would give a one-step resolution of singularities for $n\gg 0$ \cite{Yas1}. The first result obtained on this question was an affirmative answer for complex curves \cite{Yas1}.

A classical result in the theory of Nash blowups over fields of characteristic zero, due to A. Nobile, states that the Nash blowup of order one is an isomorphism if and only if the variety is non-singular \cite{Nob}. There are generalizations of this result for the higher Nash blowup in the case of normal toric varieties, normal hypersurfaces and toric curves \cite{DuarteToric,DuarteHyp,ChDG}. In contrast, it is also known that Nobile's Theorem over fields of positive characteristic fails in general: there are singular curves whose Nash blowup of any order is an isomorphism (for $n=1$ the counterexample is due to A. Nobile \cite{Nob} and for $n\geq1$, to T. Yasuda \cite{Yas1}). However, it was recently proven that Nobile's Theorem is true for normal varieties in prime characteristic \cite{DuarteNB}. From this result one may wonder if  other results only known in characteristic zero can be also obtained in prime characteristic.  In this manuscript, in particular, we are interested in results concerning toric varieties.
Our first main result states that a higher version of Nobile's Theorem holds over fields of prime characteristic for normal toric varieties.

\begin{theoremx}[{see Theorem  \ref{t. analogue Nobile}}]\label{nob toric}
Let $X$ be a normal toric variety over an algebraically closed  field $\KK$ of positive characteristic. If $\Nash_n(X)\cong X$ for some $n\geq1$, then $X$ is a non-singular variety.
\end{theoremx}

Theorem \ref{nob toric} suggests that the question regarding one-step resolution via higher Nash blowups can be reconsidered in arbitrary characteristic for normal toric varieties. R. Toh-Yama \cite{Toh} settles this question over $\CC$ giving an example of a toric surface whose every higher Nash blowup is singular. We extend this result to prime characteristic. 

\begin{theoremx}[{see Theorem \ref{ThmNotResA3}}]\label{nash A3}
Let $\KK$ be an algebraically closed field of positive characteristic.
Let $X=\V(xy-z^4)\subseteq\KK^3$. Then, $\Nash_n(X)$ is singular for all $n$.
\end{theoremx}

The theory of Nash blowups in the context of toric varieties has been an object of intense study \cite{Reb,GS1,Ataetal,GrigMil,GlezTeis,DuarSurf,DuarteToric,Toh,ChDG}. This paper aims to contribute to this study by emphasizing that the combinatorial nature of toric varieties allows us to work over fields of arbitrary characteristic. Our work also goes along with the modern approach of the theory of toric varieties over fields of arbitrary characteristic \cite{KKMS,Sturm,Liu,GlezTeis}.\\

\noindent\textbf{Convention: }Throughout this paper, $\KK$ denotes an algebraically closed field and all varieties are assumed to be irreducible. In particular, $X$ always denotes an irreducible variety over $\KK$. We denote as $\NN$ the set of non-negative integers and $\ZZ^+$ the set of positive integers.

\section*{Acknowledgments}
We thank Josep \`Alvarez Montaner, and Jack Jeffries for helpful comments. We also thank Rin Toh-yama for his careful reading of a previous version of this paper. We thank the anonymous referee for helpful comments that improved this manuscript.

\section{Higher Nash blowup of normal toric varieties}

\subsection{Higher Nash blowups}

In this subsection we recall the definition of the higher Nash blowup of an algebraic variety, and we state a basic property of these blowups.

Let $X$ be an algebraic variety of dimension $d$, $x\in X$ a $\KK-$point, and $\m$ the maximal ideal corresponding to $x$. Let $x^{(n)}:=\Spec(\Oxx/\m^{n+1})$ be the $nth$ infinitesimal neighborhood of $x$. If $x$ is a non-singular point of $X$, then $x^{(n)}$ is a closed subscheme of $X$ of length $N=\binom{d+n}{d}$. Therefore, it corresponds to a point
$$[x^{(n)}]\in\Hilb_N(X),$$
where $\Hilb_N(X)$ is the Hilbert scheme of $N$ points of $X$. Let $\Si(X)$ denote the singular locus of $X$. We have a map
\begin{align}
\delta_n:X\setminus \Si(&X)\rightarrow\Hilb_N(X),\notag\\
&x\mapsto [x^{(n)}].\notag
\end{align}

\begin{definition}[{\cite[Definition 1.2]{Yas1}}] 
The \textit{Nash blowup of order n of} $X$, denoted by $\Nash_n(X)$, is the closure of the graph of $\delta_n$ with reduced scheme structure in $X\times_{\KK}\Hilb_N(X)$. By restricting the projection $X\times_{\KK}\Hilb_N(X)\rightarrow X$ we obtain a map 
$$\pi_n:\Nash_n(X)\rightarrow X.$$
This map is projective, birational, and it is an isomorphism over $X\setminus\Si(X)$. In addition, $\Nash_1(X)$ is canonically isomorphic to the classical Nash blowup of $X$ \cite[Section 1]{Yas1}.
\end{definition}

The following theorem characterizes smoothness in terms of the usual Nash blowup.

\begin{theorem}[{\cite{Nob,DuarteNB}}]\label{Nob thm}
Let $X$ be a normal variety. Then  $\Nash_1(X)\cong X$ if and only if $X$ is non-singular.
\end{theorem}

The previous theorem holds without the assumption of normality in characteristic zero \cite{Nob}. We are interested in showing an analog of Theorem \ref{Nob thm} for $n\geq1$ in positive characteristic for normal toric varieties, thus extending the corresponding result in characteristic zero \cite{DuarteToric}.

\subsection{Combinatorial description of the normalization of the higher Nash blowup of a normal toric variety}

In order to study Theorem \ref{nob toric}, we first give a combinatorial description of the higher Nash blowup in arbitrary characteristic. Our study uses the general theory of toric varieties over arbitrary fields \cite{KKMS,Sturm,Liu,GlezTeis}. It is inspired on the combinatorial description of $F$-blowups of toric varieties in positive characteristic \cite[Proposition 3.5]{Yas2}.

\begin{notation}\label{nota toric}
Throughout this section, we use the following notation.
\begin{itemize}
\item $\sigma\subseteq\RR^d$ denotes a strictly convex rational polyhedral cone of dimension $d$. After a suitable change of coordinates, we can assume that $\sd\subseteq\RR^d_{\geq0}$.
\item  Let $\{a_1,\ldots,a_s\}\subseteq\NN^d$ be the minimal set of generators of $\sd\cap\ZZ^d$. In particular, $\KK[\sd\cap\ZZ^d]=\Su\subseteq\Po$. 
\item $X:=\Spec \Su$ denotes the corresponding $d$-dimensional normal toric variety with torus $\tor$. 
\item $J_0:=\langle x^{a_1}-1,\ldots,x^{a_s}-1\rangle\subseteq \Su$ and $J_n:=(J_0)^{n+1}$. Notice that $J_0$ is the maximal ideal corresponding to the closed point $(1,\ldots,1)\in\tor\subseteq  X$.
\end{itemize}
\end{notation}

\begin{remark}\label{disting}
Let $\eta:\Nn\to\Na$ be the normalization of $\Na$. The action of $\tor$ on $X$ induces an action on $\Na$ \cite[Section 2.2]{Yas1}. Hence, $\Na$ is a toric variety with torus $\tor\cong\pi_n^{-1}(\tor)$. In particular, $\Nn$ is also a (normal) toric variety. Let $\Sigma$ be the fan corresponding to $\Nn$. Since $\pi_n\circ\eta:\Nn\to X$ is proper and equivariant \cite[Sections 1.1, 2.2]{Yas1}, it follows that the support of $\Sigma$ is $\sigma$ \cite[Chapter 1, Theorem 8]{KKMS}.
\end{remark}

If $\ch(\KK)=0$, $\Nn$ can be described combinatorially in terms of a Gr\"obner fan \cite[Theorem 2.10]{DuarteToric}. We refer the reader to Sturmfels's book \cite{Sturm} and the first author's paper on Nash blowups of toric varieties \cite{DuarteToric} for the general theory of Gr\"obner bases and Gr\"obner fans over monomial subalgebras. Throughout this paper we use the usual notation for the basic concepts of Gr\"obner bases and Gr\"obner fans.

In this section we show that the combinatorial description for $\Nn$ in terms of a Gr\"obner fan is also valid over fields of arbitrary characteristic. The proof in this case is along the same lines  as in characteristic zero.
We include it for the sake of completeness.

\begin{definition}
Consider Notation \ref{nota toric}. Let $w\in\sigma\cap\ZZ^d$ and $f=\sum c_ux^u\in\Su$. Denote $d_w(f):=\max\{w\cdot u|c_u\neq0\}$ and $f_t:=t^{d_w(f)}f(t^{-w\cdot a_1}x^{a_1},\ldots,t^{-w\cdot a_s}x^{a_s})$. For an ideal $I\subseteq \Su$, we call the ideal $I_t:=\langle f_t|f\in I\rangle\subseteq\Su[t]$ the Gr\"obner degeneration of $I$ with respect to $w$ \cite[Section 15.8]{Eisenbud}.
\end{definition}

\begin{proposition}[{\cite[Proposition 2.7]{DuarteToric}}]\label{refine}
Consider Notation \ref{nota toric}. Let $\Sigma$ be the fan associated to $\Nn$ and $\GF$ be the Gr\"obner fan of $J_n$. Then, $\Sigma$ is a refinement of $\GF$.
\end{proposition}
\begin{proof}
Let $\sigma_1\in\Sigma$ be a cone different from $\{0\}$. Let $w\in\ZZ^d$ be a vector in the relative interior of $\sigma_1$. Then $w$ belongs to the relative interior of a unique cone $\sigma_2\in\GF$. Let $w'\neq w$ be in the relative interior of $\sigma_1$. We show that $\init_w(J_n)=\init_{w'}(J_n)$, and so, $w'\in\sigma_2$ according to the definition of $\GF$.

Let $\overline{\lambda}_w$ and $\overline{\lambda}_{w'}$ be the corresponding one-parameter subgroups. The following equality of limits in the toric variety $\Nn$ is well known (this is usually proved over $\CC$ but it is true over algebraically closed fields of arbitrary characteristic \cite[Chapter 1, Theorems 1' and 2(a)]{KKMS}):
\begin{equation}\label{limNn}
\lim_{t\to0}\overline{\lambda}_w(t)=\lim_{t\to0}\overline{\lambda}_{w'}(t)=:l\in\Nn.
\end{equation}
Since $\eta$ is an isomorphism over $\pi_n^{-1}(\tor)$, we can consider the induced one-parameter subgroups $\lambda_w:\KK^*\to\pi_n^{-1}(\tor)$ and $\lambda_{w'}:\KK^*\to\pi_n^{-1}(\tor)$ satisfying $\lambda_w=\eta\circ\overline{\lambda}_w$ and $\lambda_{w'}=\eta\circ\overline{\lambda}_{w'}$. 

Let $V\subseteq\Na$ be an open neighborhood of $\eta(l)$. Then $l\in\eta^{-1}(V)\subseteq \Nn$. By (\ref{limNn}), there exists a neighborhood $W\subseteq\KK$ of $0$ such that $\overline{\lambda}_w(W\cap\KK^*)\subseteq \eta^{-1}(V)$. It follows that $\lambda_w(W\cap\KK^*)\subseteq V$. Applying the same arguments to $\lambda_{w'}$ we obtain:
\begin{equation}\label{limNa}
\lim_{t\to0}\lambda_w(t)=\eta(l)=\lim_{t\to0}\lambda_{w'}(t).
\end{equation}

To conclude that $\init_w(J_n)=\init_{w'}(J_n)$ from this equality, we use an argument of limits of fibers in a flat family which appears in the work of several authors (see, for instance, \cite[Proposition 3.5]{Yas2}, \cite[Lecture 1]{Mac}).

We recall that $\lambda_w(t)\in\pi_n^{-1}(\tor)$ has the following explicit description for each $t\in\KK^*$ \cite[Section 2.2, Page 118]{DuarteToric}:
\begin{equation}\label{Eis}
\lambda_w(t)=\Big((t^{w\cdot a_1},\ldots,t^{w\cdot a_s}),\Spec\frac{\Su}{(J_n)_t}\Big), 
\end{equation}
where $(J_n)_t$ is the Gr\"obner degeneration of $J_n$ with respect to $w$. In addition, the family $\Spec\frac{\Su[t]}{(J_n)_t}\to\Spec \KK[t]$ is flat, the fiber over $t\in\KK^*$ is $\Spec\frac{\Su}{(J_n)_t}$ and the fiber over $0$ is $\Spec\frac{\Su}{in_w(J_n)}$ \cite[Theorem 15.17]{Eisenbud}. Since $\Spec\KK[t]$ is non-singular and one-dimensional, it follows that \cite[Proposition II-29]{EisHa}:
$$\lim_{t\to0}\lambda_w(t)=\Big(\lim_{t\to0}(t^{w\cdot a_1},\ldots,t^{w\cdot a_s}),\Spec\frac{\Su}{in_w(J_n)}\Big).$$ 
Applying the same arguments to $\lambda_{w'}$, (\ref{limNa}) implies $\init_w(J_n)=\init_{w'}(J_n)$.
\end{proof}

In order to prove that $\Sigma=\GF$ it is required the following lemma, whose proof is similar to the one of the previous proposition.

\begin{lemma}[{\cite[Lemma 2.9]{DuarteToric}}]\label{orbits}
Let $\sigma_1, \sigma_2\in\Sigma$ be such that the relative interiors of $\sigma_1$, $\sigma_2$ and $\sigma_1\cap\sigma_2$ are contained in the relative interior of some cone $\tau\in\GF$. Then $\eta(\gamma_{\sigma_1})=\eta(\gamma_{\sigma_2})=\eta(\gamma_{\sigma_1\cap\sigma_2})$, where $\gamma_{\sigma_1}$, $\gamma_{\sigma_2}$ and $\gamma_{\sigma_1\cap\sigma_2}$ are the corresponding distinguished points.
\end{lemma}
\begin{proof}
Let $w,w',w''\in\ZZ^d$ be in the relative interior of $\sigma_1$, $\sigma_2$ and $\sigma_1\cap\sigma_2$, respectively. Then, $\lim_{t\to0}\overline{\lambda}_w(t)=\gamma_{\sigma_1}$, $\lim_{t\to0}\overline{\lambda}_{w'}(t)=\gamma_{\sigma_2}$, $\lim_{t\to0}\overline{\lambda}_{w''}(t)=\gamma_{\sigma_1\cap\sigma_2}$ \cite[Chapter 1, Theorem 2]{KKMS}. By hypothesis we have $\init_w(J_n)=\init_{w'}(J_n)=\init_{w''}(J_n)$. The lemma follows using the same arguments of the proof of the previous proposition.
\end{proof}

The following theorem, in   arbitrary characteristic, follows from the previous two results in exactly the same way as done previously in characteristic zero.

\begin{theorem}[{\cite[Theorem 2.10]{DuarteToric}}]\label{t. Nash = GF}
Consider Notation \ref{nota toric}. Let $\Sigma$ be the fan associated to $\Nn$ and let $\GF$ be the Gr\"{o}bner fan of $J_n$. Then $\Sigma=\GrFan(J_n)$. 
\end{theorem}

In the case of characteristic-zero fields, the previous description was used to prove the following result.

\begin{theorem}[{\cite[Corollary 3.8]{DuarteToric}}]\label{Nobile toric positive}
Let  $X$ be a normal toric variety. If $\Nash_n(X)\cong X$, then $X$ is a non-singular variety.
\end{theorem}

The proof of this result requires the assumption on the characteristic of the field: there appear some coefficients of polynomials that might turn zero in arbitrary characteristic. In the following section we adapt the proof of the previous theorem for positive characteristic. 

\subsection{A characterization of smoothness for toric varieties}

As in Notation \ref{nota toric}, let $X$ be the normal toric variety defined by the cone $\sigma$. We show that if $X$ is singular, then $\pi_n\circ\eta$ is not an isomorphism. This implies that $\pi_n$ is not an isomorphism since $X$ is normal. By Theorem \ref{t. Nash = GF}, it is enough to show that $\GF$ is a non-trivial subdivision of $\sigma$, and so, $\pi_n\circ\eta$ is not injective.

By definition of Gr\"obner fan, we need to find $w$, $w'\in\sigma$ such that $\init_w(J_n)\neq \init_{w'}(J_n)$. This is equivalent to the following fact. Fix $w$ in the interior of $\sigma$ and let $>$ be any monomial order on $\Su$. Let $G$ be the reduced Gr\"obner basis of $J_n$ with respect to the refined order $>_{w}$. Then $\init_w(J_n)\neq \init_{w'}(J_n)$ for some $w'\in\sigma$ if and only if $\init_w(g)\neq \init_{w'}(g)$ for some $g\in G$ \cite[Section 1.2]{DuarteToric}. This is what we prove. 

Recall that $\{a_1,\ldots,a_s\}\subseteq\ZZ^d_{\geq0}$ denotes the minimal set of generators of $\sd\cap\ZZ^d$. It is known that the set $\{a_1,\ldots,a_s\}$ contains the ray generators of the edges of $\sd$ which we denote, after renumbering if necessary, by $\{a_1,\ldots,a_r\}$, as well as possibly some points in the relative interior of $\{\sum_{i=1}^{r}\lambda_ia_i|0\leq\lambda_i<1\}$ \cite[Proposition 1.2.23]{CLS}. Since $\sd$ has dimension $d$, we must have $r\geq d$. Let us assume that $\sigma$ is not a regular cone.

\begin{lemma}\label{l. h in J_n}
Let $\ch(\KK)=p>0$. In the context of Notation \ref{nota toric}, there exist $h\in J_n$ and $w$ in the relative interior of $\sigma$ such that $\lt_{>_w}(h)=(x^{a_i})^n$ for some $i\in\{1,\ldots,r\}$.
\end{lemma}
\begin{proof}
We proceed by induction on $n$. Using the binomials $x^{a_i}-1$, it is enough to show the case $n=1$. Consider the following map of $\KK$-algebras:
$$\phi:\KK[y_1,\ldots,y_s]\rightarrow\Su,\mbox{ }\mbox{ }\mbox{ }y_i\mapsto x^{a_i}.$$
Let $\overline{J_1}:=\langle y_1-1,\ldots,y_s-1 \rangle^2+\ker\phi$. Since $\sigma$ is not a regular cone, we must have $s>d$. Since $\sd$ has dimension $d$ we may assume, after renumbering if necessary, that $\{a_1,\ldots,a_d\}$ is linearly independent. Let $A$ be the matrix whose columns are $a_1,\ldots,a_d$, in this order. Let $\lambda':=A^{-1}a_{d+1}\in\QQ^d$. 
By multiplying by suitable integers and after renumbering if necessary, we obtain the following relation:
\begin{equation}\label{relation}
\lambda_1a_1+\cdots+\lambda_ta_t=\lambda_{t+1}a_{t+1}+\cdots+\lambda_{d+1}a_{d+1},
\end{equation}
where $\lambda_i\in\NN$ for all $i$, and for some $t\in\{1,\ldots,d\}$. 

Let $\bar{f}:=y_1^{\lambda_1}\cdots y_t^{\lambda_t}-y_{t+1}^{\lambda_{t+1}}\cdots y_{d+1}^{\lambda_{d+1}}\in\Poz$ and $f\in\Poy$ the polynomial it induces under the canonical homomorphism $j:\ZZ_p\hookrightarrow\KK$. By (\ref{relation}), $f\in\ker\phi.$ In particular, $f\in\ker\phi+\langle y_1-1,\ldots,y_s-1 \rangle^2=\overline{J_1}$. 

Let $\bar{h}:=\delta_1(y_1-1)+\cdots+\delta_{d+1}(y_{d+1}-1)\in\Poz$, where $\delta_i=\lambda_i\mod p$ for $i\leq t$ and $\delta_i=-(\lambda_i \mod p)$, for $i\geq t+1$. Let $\tilde{h}\in\Poy$ be the polynomial that $\bar{h}$ induces. Then $\tilde{h}$ is the linear part of the Taylor expansion of $f$ around $(1,\ldots,1)\in \KK^s$. Since $f\in\overline{J_1}$, we obtain $\tilde{h}\in\overline{J_1}$, and so $h:=\phi(\tilde{h})=\sum_{i}j(\delta_i)x^{a_i}+c\in J_1$ for some $c\in \KK$.

Now consider the following cases (recall that $r$ denotes the number of edges of $\sd$):
\begin{itemize}
\item[(1)] Suppose that $r>d$. Thus $a_1,\ldots,a_{d+1}\in\{a_1,\ldots,a_r\}$.
\begin{itemize}
\item[(1.1)] $p\nmid\lambda_i$ for some $1\leq i \leq d+1$. Then $\delta_i\neq0$ in $\ZZ_p$ and $\lt_{>_w}(h)=x^{a_i}$, for some $i\in\{1,\ldots,r\}$ and any $w\in\sigma$, as desired.
\item[(1.2)] Assume that $p|\lambda_i$ for all $i$. For each $i$, let $\lambda_i:=p^{l_i}m_i$, where $l_i\geq1$ and $p\nmid m_i$.
Then (\ref{relation}) becomes
$$p^{l_1}m_1a_1+\cdots+p^{l_t}m_ta_t=p^{l_{t+1}}m_{t+1}a_{t+1}+\cdots+p^{l_{d+1}}m_{d+1}a_{d+1}.$$
Let $l=\min_i\{l_i\}$ and $k_i:=l_i-l$. Factoring $p^l$ we obtain:
$$p^{k_1}m_1a_1+\cdots+p^{k_t}m_ta_t=p^{k_{t+1}}m_{t+1}a_{t+1}+\cdots+p^{k_{d+1}}m_{d+1}a_{d+1},$$
where $k_i=0$ for some $i$. Substituting (\ref{relation}) by this last equation takes us to the case $(1.1)$.
\end{itemize}
\item[(2)]Suppose that  $r=d$. 
\begin{itemize}
\item[(2.1)] $p\nmid\lambda_{i}$ for some $1\leq i \leq d$. Assume that $i=d$. Since $\{a_1,\ldots,a_s\}$ is the minimal set of generators of $\sd\cap\ZZ^d$, we have that $a_{d+1}=\sum_{i=1}^{d}\tau_ia_i$, for some $0\leq\tau_i<1$. Denote by $H$ the hyperplane generated by $\{a_1,\ldots,a_{d-1}\}$. Then $H\cap\sd$ is a facet of $\sd$, i.e., there exists $w\in\sigma$ such that $w^{\perp}=H$. In particular, $w\cdot a_i=0$ for $i=1,\ldots,d-1$, and $w\cdot a_d>0$. If $a_{d+1}\in H$ then $\lt_{>_w}(h)=x^{a_d}$, as desired. Otherwise, $w\cdot a_{d+1}>0$. Now we choose $w'$ sufficiently close to $w$ in the relative interior of $\sigma$ and such that $0<w'\cdot a_i<w'\cdot a_d$ and $0<w'\cdot a_i<w'\cdot a_{d+1}$ for all $i=1,\ldots,d-1$. We know that $a_{d+1}=\sum_{i=1}^{d}\tau_ia_i$, where, in particular, $0<\tau_d<1$. This fact allow us to choose $w'$ satisfying also $w'\cdot a_{d+1}<w'\cdot a_d$. Therefore $\lt_{>_{w'}}(h)=x^{a_d}$.
\item[(2.2)] $p|\lambda_i$ for all $1\leq i\leq d$ and $p\nmid \lambda_{d+1}$. Using the notation of case $(1.2)$ we obtain:
$$p^{l_1}m_1a_1+\cdots+p^{l_t}m_ta_t-p^{l_{t+1}}m_{t+1}a_{t+1}-\cdots-p^{l_d}m_da_d=\lambda_{d+1}a_{d+1}.$$
Since $\{a_1,\ldots,a_d\}$ is linearly independent, $\lambda_{d+1}\neq0$. Let $l=\min_i\{l_i\}$ and $k_i=l_i-l$. Then 
$$p^{l}\big(p^{k_1}m_1a_1+\cdots+p^{k_t}m_ta_t-p^{k_{t+1}}m_{t+1}a_{t+1}-\cdots-p^{k_d}m_da_d\big)=\lambda_{d+1}a_{d+1}.$$
Let $v:=p^{k_1}m_1a_1+\cdots+p^{k_t}m_ta_t-p^{k_{t+1}}m_{t+1}a_{t+1}-\cdots-p^{k_d}m_da_d$. Then $v=\frac{\lambda_{d+1}}{p^l}a_{d+1}$. Since $v\in\ZZ^d$ and $p\nmid\lambda_{d+1}$, it follows that $p$ divides each entry of $a_{d+1}$, a contradiction.
\item[(2.3)] $p|\lambda_i$ for all $1\leq i\leq d+1$. Proceed as in (1.2) to reduce this case to cases (2.1) or (2.2).
\end{itemize}
\end{itemize}
\end{proof}

\begin{lemma}\label{l. p<n+1}
Consider Notation \ref{nota toric}. If $m<n+1$, then $(x^{a_i}-1)^m\notin J_n$ for every $i$.
\end{lemma}
\begin{proof}
It is enough to prove the statement for $m=n$. We first show that $x^{a_i}-1\notin J_1$. Since $a_i$ is a primitive vector, there exists $j$ such that $a_{ij}\neq0$ and $p\nmid a_{ij}$. For simplicity of notation assume $i=j=1$. Suppose that $x^{a_1}-1\in J_1$, i.e.,
$$x^{a_1}-1=\sum h_{ij}(x^{a_i}-1)(x^{a_j}-1).$$
Taking derivations with respect to $x_1$ on both sides of this equation, we obtain a non-zero monomial on the left hand, and on the right hand a sum such that every summand has a factor $x^{a_i}-1$. Therefore, evaluating the resulting equation on $(1,\ldots,1)$, we obtain something different from zero on the left hand and zero on the right hand. This is a contradiction.

Let $J_0=\langle x^{a_1}-1,\ldots,x^{a_s}-1\rangle\subseteq\Su$. By definition, $J_n=(J_0)^{n+1}$. The ideal $J_0$ is the maximal ideal corresponding to $(1,\ldots,1)\in X$, which is non-singular. Let $R=\Su_{J_0}$, which is a regular local ring. By the previous paragraph, $x^{a_1}-1 \in J_0 R \setminus (J_0)^2 R$. Then $x^{a_1}-1$ is a linear combination of the minimal generators of $J_0 R$, plus possibly some other terms of higher order. This implies that, for each $n\geq1$, $(x^{a_1}-1)^n\in (J_0)^{n}R\setminus (J_0)^{n+1}R$. It follows that $(x^{a_1}-1)^n\notin J_n$.
\end{proof}

We are now ready to prove Theorem \ref{nob toric}.

\begin{theorem}\label{t. analogue Nobile}
Let $\ch(\KK)=p>0$. Let $X$ be the normal toric variety defined by $\sigma$ and $\pi_n\circ\eta:\Nn\rightarrow X$ be the normalized higher Nash blowup of $X$. If $X$ is singular, then $\Nn\not\cong X$. In particular,  if $\Nash_n(X)\cong X$ for some $n\geq1$, then $X$ is a non-singular variety.
\end{theorem}
\begin{proof}
Using Lemmas \ref{l. h in J_n} and \ref{l. p<n+1},  the theorem follows analogously to the characteristic zero case \cite[Theorem 3.7]{DuarteToric}. 
\end{proof}

\section{The $A_3$-singularity}


Thanks to  Theorem \ref{nob toric}, we may reconsider Yasuda's original question for normal toric varieties: the Nash blowup of order $n$ of a normal toric variety is non-singular for $n$ large enough? 

It was already proved by R. Toh-Yama \cite{Toh} that if $X=\V(xy-z^4)\subseteq\CC^3$, then $\Nash_n(X)$ is singular for all $n$. In this section, we extend this result to prime characteristic. The results presented in this section are based on the work of R. Toh-Yama \cite{Toh}.

\begin{notation}\label{nota a3}
Throughout this section we use the following notation:
\begin{itemize}
\item $\sigma\subseteq\RR^2$ is the cone generated by $(0,1)$ and $(4,-3)$.
\item $\sz:=\sd\cap\ZZ^2=\NN((1,0),(3,4),(1,1))$.
\item $J_n:=\Jn\subseteq\Suk$.
\item $\preceq$ denotes the monomial order on $\Suk$ defined by the matrix 
\begin{equation*}
\left(
\begin{array}{rr}
2 & -1 \\
1 & 1 \\
\end{array}
\right).
\end{equation*}
\item Let $\Gn$ be the reduced and marked Gr\"obner basis of $J_n$ with respect to $\preceq$.
\item Let $\lm(\Gn):=\{\alpha\in\ZZ^2|(g,\alpha)\in\Gn\}$
\item Let $\pol(\Gn):=\{g\in\Suk|(g,\alpha)\in\Gn\}$.
\item $X=\V(xy-z^4)\subseteq \KK^3$ (the $A_3$-singularity).
\end{itemize}
\end{notation}

According to Theorem \ref{t. Nash = GF}, the normalization of the Nash blowup of order $n$ of $X$ is determined by the Gr\"obner fan of $J_n$. It was recently shown that, if $\KK=\CC$,  this Gr\"obner fan contains a non-regular cone for each $n$. This  implies that $\Nash_n(X)$ is singular for every $n$ \cite[Theorem 2.21]{Toh}. We show that the same cone appears in $\GF$ for an arbitrary field.

\begin{definition}[{\cite[Definition 2.4]{Toh}}]\label{def pn}
For $n\in\ZZ^+$, let $\Pnn$ be the subset of $\sz$ consisting of the following points:
\\
For odd $n$, 
\begin{align}
&p_n:=(\tfrac{n+3}{2},0), 							&&   														\notag\\
&q_n^0:=(\tfrac{n+3}{2},1)+\tfrac{n-1}{2}(1,2),	&& q_n^i:=q_n^0-i(1,2)\mbox{ }\mbox{ }	(0\leq i \leq \tfrac{n-1}{2}),	\notag\\
&r_n^0:=q_n^0+(0,1),								&& r_n^j:=r_n^0+j(1,2)\mbox{ }\mbox{ }(0\leq j \leq \tfrac{n-1}{2}),	\notag\\
&s_n:=\tfrac{n+1}{2}(3,4). 						&&															\notag	
\end{align}
For even $n$, 
\begin{align}
&p_n:=(\tfrac{n+2}{2},0), 							&&   														\notag\\
&q_n^0:=(\tfrac{n+2}{2},0)+\tfrac{n}{2}(1,2), 	&& q_n^i:=q_n^0-i(1,2)\mbox{ }\mbox{ }	(0\leq i \leq \tfrac{n-2}{2}),	\notag\\
&r_n^0:=q_n^0+(0,1),								&& r_n^j:=r_n^0+j(1,2)\mbox{ }\mbox{ }(0\leq j \leq \tfrac{n}{2}),	\notag\\
&s_n:=\tfrac{n+2}{2}(3,4). 						&&															\notag	
\end{align}
\end{definition}

The following result is a key ingredient towards proving that $\Nash_n(X)$ is singular for all $n$, in the case $\KK=\CC$.

\begin{proposition}[{\cite[Proposition 2.15]{Toh}}]\label{Mn}
Consider Notation \ref{nota a3} and Definition \ref{def pn}. Assume that $\KK=\CC$. Then, $\lm(\Gn)=\Pnn$.
\end{proposition}

We also need the following property of the set $\Pnn$, which is very useful for inductive arguments.

\begin{lemma}[{\cite[Lemma 2.6]{Toh}}]\label{theta}
Consider Notation \ref{nota a3} and Definition \ref{def pn}. Let $\theta:\sz\to\sz$, $a\mapsto a+(1,1)$. Then, for $n\in\NN$, $n\geq2$,
\[\Pnn=
\left\{
\begin{array}{rll}
&\theta(\Pnu\setminus\{p_{n-1}\})\sqcup\{p_n,s_n\}, & n \mbox{ even},\\
&\theta(\Pnu\setminus\{s_{n-1}\})\sqcup\{p_n,s_n\}, & n \mbox{ odd}.
\end{array}
\right.
\]
\end{lemma}

Our first goal is to show that $\pol(\Gn)\subseteq\Suz$ in the case $\KK=\CC$. In order to prove this fact we need the following three lemmas.

\begin{lemma}\label{gn}
Consider Notation \ref{nota a3} and Definition \ref{def pn}. Let $n\in\NN$ be odd. There exists $g_n\in J_n\cap\Suz$ such that $\lc(g_n)=1$ and $\lm(g_n)=s_n$.
\end{lemma}
\begin{proof}
We know that $g_1:=u^3v^4+u-4uv+2\in J_1\cap\Suz$, $\lc(g_1)=1$, and $\lm(g_1)=(3,4)=s_1$ \cite[Proposition 2.13]{Toh}. Now define, for $n\geq3$, $g_n:=(g_1)^{\frac{n+1}{2}}$. Then $g_n\in(J_1)^{\frac{n+1}{2}}=J_n$, $g_n\in\Suz$, $\lc(g_n)=1$, and $\lm(g_n)=\frac{n+1}{2}(3,4)=s_n$.
\end{proof}

\begin{lemma}\label{hn}
Consider Notation \ref{nota a3} and Definition \ref{def pn}. Let $n\in\NN$ be even. There exists $h_n\in J_n\cap\Suz$ such that $\lc(h_n)=1$ and $\lm(h_n)=p_n$.
\end{lemma}
\begin{proof}
For $n=2$,  we take 
\begin{align}
h_2:&=u^2-4u^2v-u^3v^4+6u^2v^2+u-4uv+1\notag\\
&=(-1)(u-1)^2(u^3v^4-1)+(u^2v^2+2uv+3)(u-1)(uv-1)^2+(-uv-3)(uv-1)^3.\notag
\end{align}
Then $h_2$ satisfies the conditions of the lemma. From the previous lemma we know that $g_1=u^3v^4+u-4uv+2\in J_1$ and $\lm(g_1)=s_1$. With a computer algebra system one can check that the following polynomials satisfy the conditions of the lemma:
\begin{align}
h_4&:=g_1h_2-(u-1)(uv-1)^4=u^3-4u^5v^5+(\mbox{smaller terms}),\notag\\
h_6&:=g_1h_4-(uv-1)^4h_2=u^4-4u^8v^9+(\mbox{smaller terms}),\notag\\
h_8&:=g_1h_6-(uv-1)^4h_4=u^5-4u^{11}v^{13}+(\mbox{smaller terms}).\notag
\end{align}
Now define $h_n:=g_1h_{n-2}-(uv-1)^4h_{n-4}$ for $n\geq 10$. Assume $n=2k$. We claim:
\begin{itemize}
\item[(i)] $g_1h_{n-2}=u^{k+3}v^4+u^{k+1}-4u^{5+3(k-2)}v^{5+4(k-2)}+(\mbox{smaller terms})$.
\item[(ii)] $(uv-1)^4h_{n-4}=u^{k+3}v^4-4u^{4+5+3(k-4)}v^{4+5+4(k-4)}+(\mbox{smaller terms})$.
\item[(iii)] $h_n=u^{k+1}-4u^{5+3(k-2)}v^{5+4(k-2)}+(\mbox{smaller terms})$.
\end{itemize}
We proceed by induction on $k$. The cases $k=5,6$ can be checked with any computer algebra system. Assume the result is true for $k\geq 6$. Then, we have that
\begin{align}
g_1h_n&=(u^3v^4+u-4uv+2)(u^{k+1}-4u^{5+3(k-2)}v^{5+4(k-2)}+(\mbox{smaller terms})\notag\\
&=u^{k+4}v^4-4u^{5+3(k-1)}v^{5+4(k-1)}+u^{k+2}-4u^{1+5+3(k-2)}v^{5+4(k-2)}+(\mbox{smaller terms})\notag
\end{align}
for $k+1$. A direct computation shows that 
$$(k+4,4)\succeq(k+2,0)\succeq(5+3(k-1),5+4(k-1))\succeq(1+5+3(k-2),5+4(k-2)).$$ 
Therefore,
$$g_1h_n=u^{k+4}v^4+u^{k+2}-4u^{5+3(k-1)}v^{5+4(k-1)}+(\mbox{smaller terms}).$$
Thus, (i) holds. In exactly the same way (ii) and (iii) can be verified. Finally, we note that for $n\geq10$, $n=2k$, by definition and the claim, $h_n\in J_n\cap\Suz$, $\lc(h_n)=1$ and $\lm(h_n)=(k+1,0)=(\frac{n+2}{2},0)=p_{n}$.
\end{proof}

In the following result we use the notation from the two previous lemmas.

\begin{lemma}\label{groeb}
Consider Notation \ref{nota a3} and Definition \ref{def pn}. Let
$$
G_1:=\{g_1,(uv-1)^2,(u-1)(uv-1),(u-1)^2\},
$$
and
$$
G_2:=\{(uv-1)f|f\in G_1\setminus\{(u-1)^2\}\}\cup\{h_2,(u^3v^4-1)g_1\}.
$$

Now define recursively
$$
n \mbox{ odd: } G_n:=\{(uv-1)f|f\in G_{n-1}\setminus\{(u^3v^4-1)g_{n-2}\}\}\cup\{(u-1)h_{n-1},g_n\},
$$
and 
$$
n \mbox{ even: } G_n:=\{(uv-1)f|f\in G_{n-1}\setminus\{(u-1)h_{n-2}\}\}\cup\{(h_n,(u^3v^4-1)g_{n-1}\}.
$$
Then, for each $n\geq1$, $G_n\subseteq J_n\cap\Suz$, the elements of $G_n$ have all leading coefficient 1, and $\lm(G_n)=\Pnn$.
\end{lemma}
\begin{proof}
By construction and Lemmas \ref{gn} and  \ref{hn}, $G_n\subseteq J_n\cap\Suz$ and its elements have all leading coefficient 1. It remains to prove that $\lm(G_n)=\Pnn$. 

Let $n=1$. By Lemma \ref{gn}, $\lm(g_1)=s_1$. The leading monomials of the other three elements of $G_1$ coincides with the remaining elements of $\mathcal{P}_1$. Now let $n=2$. It is clear that $\lm((u^3v^4-1)g_1)=s_2$. By Lemma \ref{hn}, $\lm(h_2)=p_2$. For the other elements of $G_2$ we have $\lm((uv-1)f)=\lm(f)+(1,1)$. By Lemma \ref{theta}, we conclude that $\lm(G_2)=\mathcal{P}_2$.
Assume the result for $n-1$. 

If $n$ is odd, by Lemma \ref{hn}, $\lm((u-1)h_{n-1})=(1,0)+p_{n-1}=(1,0)+(\frac{n+1}{2},0)=(\frac{n+3}{2},0)=p_{n}$. By Lemma \ref{gn}, $\lm(g_n)=s_n$. For the other elements of $G_n$ we have $\lm((uv-1)f)=\lm(f)+(1,1)$. In addition, by Lemma \ref{gn}, $\lm((u^3v^4-1)g_{n-2})=(3,4)+s_{n-2}=(3,4)+\frac{n-1}{2}(3,4)=\frac{n+1}{2}(3,4)=s_{n-1}$. By Lemma \ref{theta}, we conclude that $\lm(G_n)=\mathcal{P}_n$.

If $n$ is even, $\lm(h_n)=p_n$ by Lemma \ref{hn}. By Lemma \ref{gn}, $\lm((u^3v^4-1)g_{n-1})=(3,4)+s_{n-1}=(3,4)+\frac{n}{2}(3,4)=\frac{n+2}{2}(3,4)=s_n$. For the other elements of $G_n$ we have $\lm((uv-1)f)=\lm(f)+(1,1)$. In addition, by Lemma \ref{hn}, $\lm((u-1)h_{n-2})=(1,0)+p_{n-2}=(1,0)+(\frac{n}{2},0)=(\frac{n+2}{2},0)=p_{n-1}$. By Lemma \ref{theta}, we conclude that $\lm(G_n)=\mathcal{P}_n$.
\end{proof}

\begin{proposition}\label{Gn Z}
Consider the Notation in Lemma \ref{groeb}, and assume that $\KK=\CC$. Then, $\pol(\Gn)\subseteq\Suz$.
\end{proposition}
\begin{proof}
According to Proposition \ref{Mn}, $\lm(\Gn)=\Pnn$. By Lemma \ref{groeb}, $G_n\subseteq J_n\cap\Suz$, the leading coefficient of all of its elements is 1, and $\lm(G_n)=\Pnn$. It follows that $G_n$ is a minimal Gr\"obner basis of $J_n$. Since $G_n\subseteq\Suz$ and the algorithm to turn $G_n$ into a reduced Gr\"obner basis takes place in $\Suz$, it follows that the resulting reduced Gr\"obner basis is contained in $\Suz$. By the uniqueness of the reduced Gr\"obner basis, we conclude that $\pol(\Gn)\subseteq\Suz$.
\end{proof}

This proposition is the first step towards proving that the Gr\"obner fan of $J_n$ contains a non-regular cone for each $n$. The following lemma is proved by R. Toh-Yama \cite{Toh} in the case $\KK=\CC$; however,   the proof is the same over any field. We reproduce it here for the sake of completeness.

\begin{lemma}[{\cite[Lemma 2.12(1)]{Toh}}]\label{dim}
Consider Notation \ref{nota a3} and Definition \ref{def pn}. Then,
$$\dim_{\KK}\Suk/\init_{\preceq}(J_n)=\dim_{\KK}\Suk/\langle\Pnn\rangle.$$
\end{lemma}
\begin{proof}
First,  it is known that 
$\dim_{\KK}\Suk/\init_{\preceq}(J_n)=\dim_{\KK}\Suk/J_n$  \cite[Proposition A.2.1]{duart2013}. 
Let $J_0=\langle u-1,u^3v^4-1,uv-1\rangle$. Then $J_0=\langle u-1,uv-1\rangle$, because
$$u^3v^4-1=(u^3v^3+u^2v^2+uv+1)(uv-1)-u^3v^4(u-1).$$
In addition, $\Suk_{J_0}$ is a regular local ring of dimension two since $(1,1,1)\in X$ is non-singular. There is an isomorphism of graded $\KK$-algebras:
$$\KK[x_1,x_2]\cong gr_{J_0\Suk_{J_0}}(\Suk_{J_0})\cong gr_{J_0}(\Suk),$$
where 
$x_1\mapsto[u-1\mod J_0^2]$ and $x_2\mapsto[uv-1\mod J_0^2]$.
Hence 
$$\dim_{\KK}\Suk/J_n=\dim_{\KK}\KK[x_1,x_2]/\langle x_1,x_2 \rangle^{n+1}=\frac{1}{2}(n+1)(n+2).$$
On the other hand, $\dim_{\KK}\Suk/\langle\Pnn\rangle=|\sz\setminus(\Pnn+\sz)|$. This last set corresponds to the monomials in $\Suk$ not divisible by any monomial in $\Pnn$. In addition, $|\sz\setminus(\Pnn+\sz)|=\frac{1}{2}(n+1)(n+2)$ \cite[Lemma 2.10(2)]{Toh}. This concludes the proof.
\end{proof}

\begin{notation}\label{nota a3p}
Consider the following notation:
\begin{itemize}
\item $J_n^{(0)}$ denotes the ideal $J_n\subseteq\Suc$.
\item $J_n^{(p)}$ denotes the ideal $J_n\subseteq\Suk$, where the characteristic of $\KK$ is $p>0$.
\item $\Gn^{(0)}$ denotes the Gr\"obner basis $\Gn$ of $J_n^{(0)}$ of Notation \ref{nota a3}.
\item $\Gn^{(p)}$ denotes the Gr\"obner basis $\Gn$ of $J_n^{(p)}$ of Notation \ref{nota a3}.
\item $\Gn^{(0)}\mod p:=\{(g^{(p)},\alpha)|(g,\alpha)\in\Gn^{(0)}\}$, where $g^{(p)}\in\KK[u,u^3v^4,uv]$ is defined as follows. Let $\overline{g^{(p)}}\in\ZZ_p[u,u^3v^4,uv]$ denote the polynomial $g$ whose coefficients are taken modulo $p$ (this makes sense by Proposition \ref{Gn Z}). Then $g^{(p)}$ denotes the polynomial induced by $\overline{g^{(p)}}$ via $j:\ZZ_p\hookrightarrow\KK$. Notice that $\lm(g^{(p)})=\alpha=\lm(g)$ since the coefficient of the leading monomial of $g$ is 1.
\end{itemize}
\end{notation}

\begin{proposition}\label{mod}
Consider Notation \ref{nota a3p}. Then,
$$\Gn^{(p)}=\Gn^{(0)} \mod p.$$
\end{proposition}
\begin{proof}
We want to show that $\Gn^{(0)} \mod p$ is the reduced and marked Gr\"obner basis  of $J_n^{(p)}$. We know that $\pol(\Gn^{(0)})\subseteq J_n^{(0)}$. Therefore, by construction, $\pol(\Gn^{(0)}\mod p)\subseteq J_n^{(p)}$. Since the leading coefficients of the elements of $\pol(\Gn^{(0)})$ are all 1, by  Proposition \ref{Mn},
$$\langle\Pnn\rangle=\langle \lm(\Gn^{(0)}\mod p)\rangle\subseteq \init_{\preceq}(J_n^{(p)}).$$
By Lemma \ref{dim}, it follows that $\langle \lm(\Gn^{(0)}\mod p)\rangle=\langle\Pnn\rangle=\init_{\preceq}(J_n^{(p)})$, implying that $\pol(\Gn^{(0)}\mod p)$ is a Gr\"obner basis of $J_n^{(p)}$. Since $\lm(\Gn^{(0)}\mod p)=\lm(\Gn^{(0)})$ and the support of each element of $\Gn^{(0)}\mod p$ is contained in the support of the corresponding element of $\Gn^{(0)}$, we conclude that $\Gn^{(p)}=\Gn^{(0)} \mod p.$
\end{proof}

\begin{notation}\label{nota cones}
Consider the following notation:
\begin{align}
\Cc:=&\{w\in\sigma|(\alpha-\beta)\cdot w\geq0,\mbox{ for all } (g,\alpha)\in\Gn^{(0)}, \beta\in \supp(g)\setminus\{\alpha\}\}, \notag\\
\Cp:=&\{w\in\sigma|(\alpha-\beta)\cdot w\geq0,\mbox{ for all } (g^{(p)},\alpha)\in\Gn^{(p)}, \beta\in \supp(g^{(p)})\setminus\{\alpha\}\}. \notag
\end{align}
\end{notation}

The following  statement is purely combinatorial, and so, it is independent of the characteristic of the field $\KK$.

\begin{proposition}[{\cite[Theorem 1.17]{Toh}}]
Consider Notation \ref{nota a3p} and \ref{nota cones}. Then $\Cc$ (respectively, $\Cp$) is a maximal cone of $\GrFan(J_n^{(0)})$ (respectively, $\GrFan(J_n^{(p)})$).
\end{proposition}

\begin{theorem}[{\cite[Theorem 2.21]{Toh}}]\label{main Toh}
Consider Notation \ref{nota a3p} and \ref{nota cones}. Let $n\in\ZZ^+$. Then $\Cc$ is a non-regular cone of $\GrFan(J_n^{(0)})$. In particular, $\Nash_n(X)$ is a singular variety, where $X\subseteq\CC^3$.
\end{theorem}

Our goal is to prove that $\Cp$ is a non-regular cone. Notice that $\Cc\subseteq\Cp$ by definition and Proposition \ref{mod}. However, this inclusion is not enough to show that $\Cp$ is non-regular. To that end, we actually prove that $\Cc=\Cp$. 

Notice that the following statements are purely combinatorial, therefore, they are independent of the characteristic of the field $\KK$.

\begin{lemma}[{\cite[Lemmas 1.16,2.16]{Toh}}]\label{crit rays}
Consider Notation \ref{nota a3p} and \ref{nota cones}. Then
\begin{enumerate}
\item $(2,-1)\in C_{\Gn^{(p)}}$.
\item Let $w\in C_{\Gn^{(p)}}\cap\ZZ^2$, $w\neq(0,0)$. If there exists $(g^{(p)},\alpha)\in\Gn^{(p)}$ and $\beta\in \supp(g^{(p)})\setminus\{\alpha\}$ such that $(\alpha-\beta)\cdot w=0$, then $w$ is a ray of $C_{\Gn^{(p)}}$.
\end{enumerate}
\end{lemma}

In the work of R. Toh-Yama \cite[Propositions 2.18,2.20]{Toh}, the generating rays of $\Cc$ are given explicitly. A similar strategy, along with Proposition \ref{mod}, gives the corresponding result for $\Cp$.

\begin{proposition}[{\cite[Propositions 2.18,2.20]{Toh}}]\label{PropToh218220}
Consider Notation \ref{nota a3p} and \ref{nota cones}. Let $L_1:=\RR_{\geq0}(2,-1)$ and 
\[L_2:=
\left\{
\begin{array}{rll}
&\RR_{\geq0}(2n-2,-n+2), & n \mbox{ odd},\\
&\RR_{\geq0}(2n,-n+1), & n \mbox{ even}.
\end{array}
\right.
\]
Then $L_1$ and $L_2$ are the rays of $\Cp$.
\end{proposition}
\begin{proof}
We start with $L_1$. By Lemma \ref{crit rays}, $(2,-1)\in\Cp$. Now consider the polynomial $f_n:=(uv-1)^{n-1}g_1\in J_n^{(0)}$ (recall that $g_1=u^3v^4+u-4uv+2$). In the proof of \cite[Proposition 2.18]{Toh} it is shown that $(f_n,(uv)^{n-1}u^3v^4)\in\Gn^{(0)}$. In addition, $(uv)^{n-1}u$ is a monomial of $f_n$ and it has 1 as coefficient. By Proposition \ref{mod}, $(f_n^{(p)},(uv)^{n-1}u^3v^4)\in\Gn^{(p)}$ and $(uv)^{n-1}u$ is in the support of $f_n^{(p)}$. Since 
$$[((n-1)(1,1)+(3,4))-((n-1)(1,1)+(1,0))]\cdot(2,-1)=0,$$ 
we conclude that $L_1$ is a ray of $\Cp$ by Lemma \ref{crit rays}. 

Now consider:
\[l_n:=
\left\{
\begin{array}{rll}
&(2n-2,-n+2), & n \mbox{ odd},\\
&(2n,-n+1), & n \mbox{ even}.
\end{array}
\right.
\]
A direct computation shows that $l_n\in\sigma$. In the proof of \cite[Proposition 2.20]{Toh} it is shown that $l_n\cdot(\alpha-\beta)\geq0$ for any $(g,\alpha)\in\Gn^{(0)}$ and any $\beta\in \supp(g)\setminus\{\alpha\}$. In particular, by Proposition \ref{mod}, the same statement holds for the elements of $\Gn^{(p)}$. It follows that $l_n\in\Cp$. To prove that $l_n$ defines a ray of $\Cp$ we proceed as in the previous paragraph.  There exists an element $(g,\alpha)\in\Gn^{(0)}$ such that:
$\alpha=q_n^{\frac{n-1}{2}}  \; \&  \; r_n^{\frac{n-1}{2}}\in \supp(g)$ if $n$ is odd, and 
$\alpha=p_n \; \& \;  s_{n-1}\in \supp(g)$ if $n$  even.
In addition, in the same work it is shown that both monomials $r_n^{\frac{n-1}{2}}\in \supp(g)$ ($n$ odd) and $s_{n-1}\in \supp(g)$ ($n$ even) have -1 as coefficient. As before, it follows that $(g^{(p)},\alpha)\in\Gn^{(p)}$, $r_n^{\frac{n-1}{2}}\in \supp(g^{(p)})\setminus\{\alpha\}$ if $n$ is odd and $s_{n-1}\in \supp(g^{(p)})\setminus\{\alpha\}$ if $n$  is even. Finally, $l_n\cdot(q_n^{\frac{n-1}{2}}-r_n^{\frac{n-1}{2}})=0$ and $l_n\cdot(p_n-s_{n-1})=0$ \cite[Lemma 2.10(7)]{Toh}. By Lemma \ref{crit rays}, we conclude that $L_2$ is a ray of $\Cp$.
\end{proof}

All previous results have as a consequence:

\begin{theorem}\label{ThmNotResA3}
Let $n\in\ZZ^+$. Then $\Cp$ is a non-regular cone of $\GrFan(J_n^{(p)})$. In particular, $\Nash_n(X)$ is a singular variety, where
$X=\V(xy-z^4)\subseteq\KK^3$.
\end{theorem}
\begin{proof}
Combining work of R. Toh-Yama  \cite[Propositions 2.18,2.20]{Toh}  and  Proposition \ref{PropToh218220}  we deduce that $\Cc=\Cp$. By Theorem \ref{main Toh}, we conclude that $\Cp$ is a non-regular cone.
\end{proof}


\bibliographystyle{alpha}
\bibliography{References}

\newcommand{\etalchar}[1]{$^{#1}$}
\def\cprime{$'$} \def\cprime{$'$} \def\cprime{$'$}
\begin{thebibliography}{CMDGF20}

\bibitem[ALP{\etalchar{+}}11]{Ataetal}
Atanas Atanasov, Christopher Lopez, Alexander Perry, Nicholas Proudfoot and
  Michael Thaddeus,
\newblock Resolving toric varieties with {N}ash blowups,
\newblock { Exp. Math.}, 20 (2011), no. 3, 288--303.

\bibitem[CLS11]{CLS}
David~A. Cox, John~B. Little and Henry~K. Schenck,
\newblock { Toric varieties}, volume 124 of {  Graduate Studies in
  Mathematics},
\newblock American Mathematical Society, Providence, RI, 2011.

\bibitem[CMDGF20]{ChDG}
Enrique Chavez~Martinez, Daniel Duarte and Arturo Giles~Flores,
\newblock A higher-order tangent map and a conjecture on the higher {N}ash blowup
  of curves,
\newblock { arXiv:1803.04595, to appear in Math. Z.}, 2020.

\bibitem[DNB20]{DuarteNB}
Daniel Duarte and Luis N{\'u}{\~n}ez-Betancourt,
\newblock Nash blowups in prime characteristic,
\newblock { arXiv:2001.10491}, 2020.

\bibitem[Dua13]{duart2013}
Daniel Duarte,
\newblock { Nash modification on toric surfaces and higher Nash blowup on
  normal toric varieties},
\newblock PhD thesis, 2013.
\newblock These de doctorat dirigee par Spivakovsky, Mark Mathematiques
  fondamentales Toulouse 3, 2013.

\bibitem[Dua14a]{DuarteToric}
Daniel Duarte,
\newblock Higher {N}ash blowup on normal toric varieties,
\newblock { J. Algebra}, 418 (2014), 110--128.

\bibitem[Dua14b]{DuarSurf}
Daniel Duarte,
\newblock Nash modification on toric surfaces,
\newblock { Rev. R. Acad. Cienc. Exactas F\'{\i}s. Nat. Ser. A Mat. RACSAM},
  108 (2014), no. 1, 153--171.

\bibitem[Dua17]{DuarteHyp}
Daniel Duarte,
\newblock Computational aspects of the higher {N}ash blowup of hypersurfaces,
\newblock { J. Algebra}, 477 (2017), 211--230.

\bibitem[EH00]{EisHa}
David Eisenbud and Joe Harris,
\newblock { The geometry of schemes}, { Graduate Texts in
  Mathematics}, 197, 
\newblock Springer-Verlag, New York, 2000.

\bibitem[Eis95]{Eisenbud}
David Eisenbud,
\newblock { Commutative algebra with a view toward algebraic geometry}, { Graduate Texts in
  Mathematics}, 150,
\newblock Springer-Verlag, New York, 1995.

\bibitem[GM12]{GrigMil}
Dima Grigoriev and Pierre~D. Milman,
\newblock Nash resolution for binomial varieties as {E}uclidean division, {A}
  priori termination bound, polynomial complexity in essential dimension {$2$},
\newblock { Adv. Math.}, 231 (2012), no. 6, 3389--3428.

\bibitem[GPT14]{GlezTeis}
Pedro~D. Gonz\'{a}lez~P\'{e}rez and Bernard Teissier,
\newblock Toric geometry and the {S}emple-{N}ash modification,
\newblock { Rev. R. Acad. Cienc. Exactas F\'{\i}s. Nat. Ser. A Mat. RACSAM},
  108 (2014), no. 1, 1--48.

\bibitem[GS77]{GS1}
Gerardo Gonzalez~Sprinberg,
\newblock Transform\'{e} de {N}ash et \'{e}ventail de dimension {$2$},
\newblock { C. R. Acad. Sci. Paris S\'{e}r. A-B}, 284(1977), no. 1, A69--A71.

\bibitem[KKMSD73]{KKMS}
G.~Kempf, Finn~Faye Knudsen, D.~Mumford and B.~Saint-Donat,
\newblock { Toroidal embeddings {1}},
\newblock Lecture Notes in Mathematics, Vol. 339. Springer-Verlag, Berlin-New
  York, 1973.

\bibitem[Liu13]{Liu}
Qing Liu,
\newblock Three hours with toric varieties,
\newblock unpublished notes, 2013.

\bibitem[Mac07]{Mac}
Diane Maclagan,
\newblock Notes on hilbert schemes,
\newblock unpublished notes, 2007.

\bibitem[Nob75]{Nob}
A.~Nobile.
\newblock Some properties of the {N}ash blowing-up.
\newblock {\em Pacific J. Math.}, 60 (1975), no. 1, 297--305.

\bibitem[Reb77]{Reb}
Vaho Rebassoo,
\newblock Some properties of the {N}ash blowing-up,
\newblock ProQuest LLC, Ann Arbor, MI, 1977,
\newblock Thesis (Ph.D.)--University of Washington.


\bibitem[Stu96]{Sturm}
Bernd Sturmfels,
\newblock { Gr\"{o}bner bases and convex polytopes}, volume~8 of { 
  University Lecture Series},
\newblock American Mathematical Society, Providence, RI, 1996.

\bibitem[TY19]{Toh}
Rin Toh-Yama,
\newblock Higher {N}ash blowups of the {$A_3$}-singularity,
\newblock { Comm. Algebra}, 47 (2019), no. 11, 4541--4564.

\bibitem[Yas07]{Yas1}
Takehiko Yasuda,
\newblock Higher {N}ash blowups,
\newblock { Compos. Math.}, 143 (2007), no. 6, 1493--1510.

\bibitem[Yas12]{Yas2}
Takehiko Yasuda,
\newblock Universal flattening of {F}robenius,
\newblock { Amer. J. Math.}, 134 (2012), no. 2, 349--378.

\end{thebibliography}

\end{document}